\documentclass[11pt,reqno]{amsart}
\usepackage[tmargin=1in,bmargin=1in,rmargin=1in,lmargin=1in]{geometry}
\usepackage[mathscr]{eucal}
\usepackage[breaklinks=true]{hyperref}

\usepackage{xcolor}
\usepackage{graphicx}
\usepackage{amsmath,amsthm,amsfonts,amssymb,tocvsec2}
\usepackage{mathdots}
\usepackage{mathtools}
\usepackage{caption}
\usepackage{subcaption}
\usepackage{float}
\usepackage{multicol}

\usepackage{algorithm}
\usepackage{algpseudocode}

\usepackage{listings}
\theoremstyle{plain}
\newtheorem{theorem}{Theorem}[section]

\newtheorem{lemma}[theorem]{Lemma}

\newtheorem{utheorem}{\textrm{\textbf{Theorem}}}

\theoremstyle{definition}
\newtheorem{defn}[theorem]{Definition}

\newtheorem{rem}[theorem]{Remark}

\numberwithin{equation}{section}

\newcommand{\ceil}[1]{\left\lceil #1 \right\rceil}

\begin{document}
	\title[Constructing strictly sign regular matrices of all sizes and sign patterns]{Constructing strictly sign regular matrices of all sizes and sign patterns}
	
	\author{Projesh Nath Choudhury and Shivangi Yadav}
	\address[P.N.~Choudhury]{Department of Mathematics, Indian Institute of Technology Gandhinagar, Gujarat 382355, India}
	\email{\tt projeshnc@iitgn.ac.in}
	\address[S.~Yadav]{Department of Mathematics, Indian Institute of Technology Gandhinagar, Gujarat 382355, India; Ph.D. Student}
	\email{\tt shivangi.yadav@iitgn.ac.in, shivangi97.y@gmail.com}
	
	\date{\today}
	
	\begin{abstract}
		The class of strictly sign regular (SSR) matrices has been extensively studied by many authors over the past century, notably by Schoenberg, Motzkin, Gantmacher, and Krein. A classical result of Gantmacher--Krein assures the existence of SSR matrices for any dimension and sign pattern. In this article, we provide an algorithm to explicitly construct an SSR matrix of any given size and sign pattern. (We also provide in an Appendix, a Python code implementing our algorithm.) To develop this algorithm, we show that one can extend an SSR matrix by adding an extra row (column) to its border, resulting in a higher order SSR matrix. Furthermore, we show how  inserting a suitable new row/column between any two successive rows/columns of an SSR matrix results in a matrix that remains SSR. We also establish analogous results for strictly sign regular $m \times n$ matrices of order $p$ for any $p \in [1, \min\{m,n\}]$.
	\end{abstract}
	
	\subjclass[2020]{15B48, 15A83, 15A15, 15--04}
	
	\keywords{Strictly sign regular matrices, matrix completion, SSR construction algorithm}
	
	\maketitle
	
	\vspace*{-11mm}
	\settocdepth{section}
	\tableofcontents
	
	\section{Introduction}
	
	Let $m,n\geq p \geq1$ be integers. An $m\times n$ matrix $A$ is said to be \textit{strictly sign regular of order $p$} (SSR$_p$) if all its minors of size $k\leq p$ have the identical sign $\epsilon_k\in\{\pm1\}$. We say $A$ is \textit{strictly sign regular} (SSR) if $A$ is SSR$_p$ for $p = \min\{m,n\}$. If the signs of the minors are all positive, then the matrix $A$ is called \textit{totally positive} (TP). If one replaces `minor' with `non-zero minor' in the preceding definitions, this yields \textit{sign regular matrices} (SR) (respectively \textit{totally non-negative matrices} (TN)). 
	
	This class of matrices was first studied by Schoenberg in 1930 to characterize the variation diminishing property. Variation diminution is a fundamental property -- and was used by Motzkin \cite{Mot36}, Gantmacher--Krein \cite{GK50}, and Brown--Johnstone--MacGibbon \cite{BJM81} to characterize SSR and SR matrices with various technical constraints: (i)~on the size of the matrix, (ii)~on its rank, or (iii)~on the sign sequences considered. In \cite{CY-VDP23}, we removed all these constraints and in a sense ended this line of investigation by: (i)~providing a complete characterization when an arbitrary matrix is SSR -- including with a specific sign pattern, (ii)~strengthening Motzkin's characterization of SR matrices without any rank constraint, and (iii)~providing a single-vector test for detecting the (strict) sign regular property of a matrix using variation diminution. We continued our study of SSR and SR matrices in \cite{CY-LPP24}, where we classified all surjective linear operators that preserve these matrices -- with arbitrary and for every fixed sign pattern. (In the special case of TP matrices, the parallel question of understanding their entrywise preservers was recently resolved in \cite{ BGKP20}.) All of these works require many test TP and SSR matrices, which makes explicitly finding/constructing such examples a fundamental question. This is what motivated the present paper.
	
	SSR matrices (including TP matrices) have long been studied in numerous theoretical and applied branches of mathematics, including analysis, cluster algebras, combinatorics, Gabor analysis, interpolation theory and splines, matrix theory, probability and statistics, and representation theory \cite{BGKP20,BFZ96,Bre95,FP12,FZ02,GRS18,K68,Karlinsplines,KW14}. Given their rich applications, it naturally raises the following questions: (i)~Can one construct a strictly sign regular matrix of any size and sign pattern? (ii)~Can a row/column be added to a strictly sign regular matrix so that the resulting matrix is again SSR? While Gantmacher--Krein \cite{GK50} explained why SSR matrices of all sizes and all sign patterns exist, and that they are dense in the SR matrices of the same size, they did not explain (nor could we find later works in the literature) how to construct explicit examples of such matrices -- i.e., the two questions above.
	
	This paper provides a positive answer to both of these questions. We propose an algorithm that outputs an SSR matrix for any specified size and sign pattern. (Such algorithms have been carried out for other classes of positive matrices; see e.g.\ \cite{AFJPS24, TZ21}.) Our algorithm is based on establishing the fact that every SSR matrix can be embedded in a higher order SSR matrix ``touching'' one of the four borders. We now formally state this result; in it and henceforth, we will use \textit{`line'} to indicate either a row or column of the matrix.
	
	\begin{utheorem}\label{A}
		Given integers $m,n\geq1$ and an $m\times n$ SSR matrix, it is possible to add a line to any of its borders such that the resulting matrix remains SSR. If minors of a larger size occur, they can be made either all positive or all negative.
	\end{utheorem}
	
	This theorem shows that an SSR matrix can always be extended by adding a line at any one or all four of its borders. It also reveals that, upon adding a line, if any minor of a new size appears, the choice of its sign is in our hands. We prove Theorem~\ref{A} in Section~\ref{Section_SSR_boundary_line_insertion}, and subsequently use this result to provide an algorithm for constructing an SSR matrix in Section~\ref{Section_SSR_Construction}. Moreover, Theorem~\ref{A} leads to the question: Can we extend an SSR matrix by adding a new column/row between any two successive columns/rows? We address this affirmatively in the following theorem.
	
	\begin{utheorem}\label{B}
		Let $m,n\geq1$ be integers and $\epsilon=(\epsilon_1,\ldots,\epsilon_{\min\{m,n\}})$ be a sign pattern. Then, a line can be inserted between any two consecutive columns (or rows) of an $m\times n$ SSR$(\epsilon)$ matrix (see Definition~\ref{defn1}--(v)) to form an $m\times(n+1)$ (or $(m+1)\times n$) SSR$(\epsilon^{\prime})$ matrix where $\epsilon^{\prime}_i=\epsilon_i$ for $i=1,\ldots,\min\{m,n\}$. If minors of size $\min\{m,n\}+1$ occur, then $\epsilon_{\min\{m,n\}+1}^{\prime}$ can be made either positive or negative.
	\end{utheorem}
	
	Note that the above theorem holds for every sign pattern, particularly the TP case when $\epsilon_i=1$ for all $i=1,\ldots,\min\{m,n\}$. (This case was studied in \cite{Smith2000} in 2000.) We prove Theorem~\ref{B} in Section~\ref{Section_line_insertion_middle}.
	
	\begin{rem}
		It is natural to ask for SSR$_p$ analogues of the above main results. In Sections~\ref{Section_SSR_boundary_line_insertion}~and~\ref{Section_line_insertion_middle}, we indeed prove similar results for strictly sign regular matrices of order $p$, following their SSR counterparts. See Remark~\ref{Remark_SSR_p_border} and Theorem~\ref{Theorem_SSR_k_middle}.
	\end{rem}
	
	Also note that the analogous question for the ``non-strict'' SR matrices has a trivial answer: one can insert a zero row or column or clone any row or column.
	
	We conclude this section on a philosophical note. The numerous applications of totally positive matrices in several branches of pure and applied mathematics stem from their vast collection of examples. The main results of this paper help construct examples of strictly sign regular matrices of arbitrary sizes and sign patterns. We hope these examples will lead to a wide range of applications in as broad a spectrum as the restricted subclass of totally positive matrices.
	\bigskip
	
	\section{Line insertion at the borders}\label{Section_SSR_boundary_line_insertion}
	
	This section shows that a new SSR matrix can be constructed from any given SSR matrix by adding a line to its border. Furthermore, we show that if square submatrices of a new size appear after inserting a line, we can control the sign of their determinants, making them either all positive or all negative. We start by recording  some definitions and notations.
	\begin{defn}\label{defn1}
		Let $m,n\geq1$ be integers throughout this manuscript.
		\begin{itemize}
			\item[(i)] A \textit{contiguous submatrix} is one whose rows and columns are indexed by consecutive positive integers, and a \textit{contiguous minor} is the determinant of a contiguous square submatrix.
			
			\item[(ii)] An anti-diagonal matrix whose all anti-diagonal elements are one is called an \textit{exchange matrix}. We denote an $n\times n$ exchange matrix by $P_n$:
			\[P_n:=\begin{pmatrix}
				0 & \cdots & 1 \\
				\vdots & \iddots & \vdots \\
				1 & \cdots & 0 \\
			\end{pmatrix} \in \{ 0, 1 \}^{n\times n}.\]
			
			\item[(iii)] $[n]$ represents the set of the first $n$ positive integers. For $m\leq n$, define $[m,n]:=\{m,m+1,\ldots,n\}$.
			
			\item[(iv)] For an $m\times n$ matrix $A$ and index sets $I\subseteq[m]$, $J\subseteq[n]$, let $A_{I\times J}$ denote the submatrix of $A$ whose rows and columns are indexed by $I$ and $J$, respectively.
			
			\item[(v)] We write SSR$(\epsilon)$ to indicate that a matrix is SSR with the sign pattern $\epsilon$.
			
		\end{itemize}
	\end{defn}
	
	We next prove a lemma that is crucial in showing our main results.
	
	\begin{lemma}\label{SSR_col_relation_general}
		Let $A=[\boldsymbol{c}^1|\cdots |\boldsymbol{c}^n]\in\mathbb{R}^{(n-1)\times n}$ be an SSR$(\epsilon)$ matrix for $\epsilon \in \{\pm 1\}^{n-1}$. Then, for $1\leq k\leq n$, the $k^{th}$ column of $A$ can be expressed as a real linear combination of its remaining columns of the form:
		\begin{equation}\label{SSR_col_relation_form}
			\boldsymbol{c}^k=\displaystyle{\sum_{\substack{i=1 \\ i\neq k}}^{n}}x_i\boldsymbol{c}^i, ~ \text{where sign}(x_i)=\begin{cases}
				(-1)^i, ~ &\text{if k is odd}, \\
				(-1)^{i-1}, ~ &\text{if k is even}.
			\end{cases}
		\end{equation}
	\end{lemma}
	\begin{proof}
		Let $A=[\boldsymbol{c}^1|\cdots |\boldsymbol{c}^n]$ be an SSR$(\epsilon)$ matrix of size $(n-1)\times n$. Since rank$(A)=n-1$, any column of $A$ can be written as a linear combination of its remaining $n-1$ columns. Fix $k\in\{1,\ldots,n\}$. Then there exist scalars $x_1, \ldots, x_{k-1}, x_{k+1}, \ldots, x_n$ such that
		\begin{equation*}\label{SSR_linear_comb_general}
			\boldsymbol{c}^k = x_1\boldsymbol{c}^1+\cdots+ x_{k-1}\boldsymbol{c}^{k-1}+ x_{k+1}\boldsymbol{c}^{k+1}+ \cdots+ x_n\boldsymbol{c}^n.
		\end{equation*}
		Let $i<k$. Solving the above system of linear equations using Cramer's rule for $x_i$, we obtain
		\begin{align}\label{SSR_xi_sign}
			x_i &= \frac{\det\left[\boldsymbol{c}^1| \cdots| \boldsymbol{c}^{i-1}| \boldsymbol{c}^k| \boldsymbol{c}^{i+1}| \cdots | \boldsymbol{c}^{k-1}| \boldsymbol{c}^{k+1}| \cdots| \boldsymbol{c}^n\right]}{\det\left[\boldsymbol{c}^1| \cdots | \boldsymbol{c}^{k-1}| \boldsymbol{c}^{k+1}| \cdots | \boldsymbol{c}^n\right]}. \nonumber \\ \nonumber \\
			\implies x_i &= \frac{(-1)^{k-i-1}\det\left[\boldsymbol{c}^1| \cdots| \boldsymbol{c}^{i-1}| \boldsymbol{c}^{i+1}| \cdots| \boldsymbol{c}^{k-1}| \boldsymbol{c}^k| \boldsymbol{c}^{k+1}| \cdots| \boldsymbol{c}^n\right]}{\det\left[\boldsymbol{c}^1| \cdots| \boldsymbol{c}^{k-1}| \boldsymbol{c}^{k+1}| \cdots| \boldsymbol{c}^n\right]}.
		\end{align}
		As the sign of both determinants in \eqref{SSR_xi_sign} is identical, sign($x_i$)$=(-1)^{k-i-1}$. Hence,
		\begin{align*}
			\text{sign}(x_i)=\begin{cases}
				(-1)^i, ~ &\text{if $k$ is odd}, \\
				(-1)^{i-1}, ~ &\text{if $k$ is even}.
			\end{cases}
		\end{align*}
		The argument is similar for $i>k$.
	\end{proof}
	
	The above lemma gives an essential relationship among the columns of an $(n-1)\times n$ SSR$(\epsilon)$ matrix: each such column can be expressed as an alternating linear combination of its remaining columns, beginning with positive coefficients of the adjacent columns.
	\begin{rem}
			The converse of the above lemma does not hold. This can be verified by considering the following matrix:
			\begin{align*}
				A = [\boldsymbol{c}^1| \boldsymbol{c}^2| \boldsymbol{c}^3| \boldsymbol{c}^4] = \begin{pmatrix}
					10 & 1 & 3 & 6 \\
					1 & 1 & 2 & 1 \\
					1 & 2 & 3 & 1
				\end{pmatrix}.
			\end{align*}
			Note that $\boldsymbol{c}^1 = \boldsymbol{c}^2 -\boldsymbol{c}^3 + 2\boldsymbol{c}^4$, $\boldsymbol{c}^2 = \boldsymbol{c}^1 +\boldsymbol{c}^3 - 2\boldsymbol{c}^4$, $\boldsymbol{c}^3 = -\boldsymbol{c}^1 +\boldsymbol{c}^2 + 2\boldsymbol{c}^4$, and $\boldsymbol{c}^4 = \frac{1}{2}\boldsymbol{c}^1 -\frac{1}{2}\boldsymbol{c}^2 + \frac{1}{2}\boldsymbol{c}^3$. Thus each column of $A$ satisfies the relation~\eqref{SSR_col_relation_form} as well the submatrix $[\boldsymbol{c}^2| \boldsymbol{c}^3| \boldsymbol{c}^4]$ is SSR, but the matrix $A$ is not SSR.
	\end{rem}
	
	Therefore, it is natural to ask: Given an SSR$(\epsilon)$ matrix $A\in\mathbb{R}^{n\times n}$, can we add a new column to $A$ using a linear combination of its $n$ columns, such that the resulting $n\times(n+1)$ matrix remains SSR$(\epsilon)$? More generally, one can ask: Given any $m\times n$ SSR$(\epsilon)$ matrix, can we embed a column (row) between any two successive columns (rows), or at any border, to form a new SSR matrix? To answer these questions, we first study the problem of line insertion at the borders of an SSR matrix, which constitutes the main result of this section. This requires the following well-known result by Karlin.
	
	\begin{theorem}[Karlin \cite{K68}] \label{Theorem_Karlin}
		Suppose $m,n\geq p\geq 1$ are integers and $A\in\mathbb{R}^{m\times n}$. Then $A$ is an SSR$_p(\epsilon)$ matrix with the sign pattern $\epsilon=(\epsilon_1,\ldots,\epsilon_p)$ if and only if all $k\times k$ contiguous minors of $A$ have sign $\epsilon_k$, for every $k\in[p]$.
	\end{theorem}
	
	 We begin with square SSR matrices.
	
	\begin{theorem}\label{Theorem_square_line_insertion}
		A line can be added to any border of a square SSR$(\epsilon)$ matrix such that the resulting matrix remains SSR$(\epsilon)$.
	\end{theorem}
	\begin{proof}
		Let $\epsilon=(\epsilon_1,\ldots,\epsilon_n)$ be a given sign pattern and let $A=[\boldsymbol{c}^1| \cdots| \boldsymbol{c}^n]\in\mathbb{R}^{n\times n}$ be an SSR$(\epsilon)$ matrix. Note that it suffices to show that one can add an extra column either to the left or right border of $A$ to obtain an $n\times(n+1)$ SSR$(\epsilon)$ matrix. The result for rows then follows by taking transposes. We first show how to add a column $\boldsymbol{c}\in\mathbb{R}^{n}$ to the left border of $A$ such that $\widehat{A}:=[\boldsymbol{c}| A]$ is SSR$(\epsilon)$. By Lemma~\ref{SSR_col_relation_general}, $\boldsymbol{c}$ must have the following form:
		\[\boldsymbol{c}=\displaystyle{\sum_{i=1}^{n}}(-1)^{i-1}y_i\boldsymbol{c}^i, \hspace{.1cm} \mathrm{where} \hspace{.15cm} y_i>0 \hspace{.1cm} \mathrm{for} \hspace{.15cm} i=1,\ldots,n.\]
		Given an integer $1\leq k\leq n$, note that the scalars $y_1,\ldots,y_{k-1}$ do not appear in the determinant of any $k\times k$ contiguous submatrices of $\widehat{A}$ that include its first column. At the same time, we get contributions from $y_k,\ldots,y_n$. Based on this observation, we outline the proof strategy: We begin by considering an $n\times n$ contiguous submatrix of $\widehat{A}$ containing its first column. From this submatrix, we first choose $y_n>0$ such that the sign of its determinant is $\epsilon_n$. Next, for any $1\leq k\leq n-1$ and for $k\times k$ contiguous submatrices of $\widehat{A}$ that include its first column, we show that we can select $y_{k}>0$ such that the sign of the determinant of these submatrices is $\epsilon_k$, assuming we already know the values of $y_n, y_{n-1}, \ldots, y_{k+1} > 0$. This completes our proof by Theorem~\ref{Theorem_Karlin}.
		
		To pick $y_n>0$, consider the following minor of $\widehat{A}$.
		\begin{align*}
			\det(\widehat{A}_{[n]\times[n]})&=\det\left[\boldsymbol{c}| \boldsymbol{c}^1| \cdots| \boldsymbol{c}^{n-1}\right] \\
			&=\displaystyle{\sum_{i=1}^{n}}(-1)^{i-1}y_i\det\left[\boldsymbol{c}^i| \boldsymbol{c}^1| \cdots| \boldsymbol{c}^{n-1}\right] \\
			&=(-1)^{n-1}y_n\det\left[\boldsymbol{c}^n| \boldsymbol{c}^1| \cdots| \boldsymbol{c}^{n-1}\right] \\
			&=y_n\det\left[\boldsymbol{c}^1| \cdots| \boldsymbol{c}^n\right].
		\end{align*}
		Hence, sign$(\det(\widehat{A}_{[n]\times[n]}))=\epsilon_{n}$ for every $y_n>0$. Choose any $y_n>0$; having defined $y_n,y_{n-1},\ldots,y_{k+1}>0$ for some $1\leq k\leq n-1$, we inductively define $y_k$ as follows: Let $I\subsetneq[n]$ be a contiguous set with $|I|=k$. Then for any real $y_k$,
		\begin{align*}
			\det(\widehat{A})_{I\times[k]}&=\det\left[\boldsymbol{c}| \boldsymbol{c}^1| \cdots| \boldsymbol{c}^{k-1}\right]_{I\times[k]} \\
			&=\det\left[\displaystyle{\sum_{i=k}^{n}}(-1)^{i-1}y_i\boldsymbol{c}^i| \boldsymbol{c}^1| \cdots| \boldsymbol{c}^{k-1}\right]_{I\times[k]} \\
			&= \det\left[(-1)^{k-1}y_{k}\boldsymbol{c}^{k}| \boldsymbol{c}^1| \cdots| \boldsymbol{c}^{k-1}\right]_{I\times[k]} +\det\left[\displaystyle{\sum_{i=k+1}^{n}}(-1)^{i-1}y_i\boldsymbol{c}^i| \boldsymbol{c}^1| \cdots| \boldsymbol{c}^{k-1}\right]_{I\times[k]} \\
			&= y_k\det\left[\boldsymbol{c}^1| \cdots| \boldsymbol{c}^{k-1}| (-1)^{2(k-1)} \boldsymbol{c}^{k}\right]_{I\times[k]} +\displaystyle{\sum_{i=k+1}^{n}(-1)^{i-1}}y_i\det\left[\boldsymbol{c}^i| \boldsymbol{c}^1| \cdots| \boldsymbol{c}^{k-1}\right]_{I\times[k]}  \\
			&= y_k\det\left[\boldsymbol{c}^1| \cdots| \boldsymbol{c}^{k}\right]_{I\times[k]} +\displaystyle{\sum_{i=k+1}^{n}(-1)^{i+k-2}}y_i\det\left[\boldsymbol{c}^1| \cdots| \boldsymbol{c}^{k-1}| \boldsymbol{c}^i\right]_{I\times[k]}.
		\end{align*}
		We require sign$(\det(\widehat{A})_{I\times[k]})=\epsilon_k$. Hence, we must ensure that $\epsilon_k\det(\widehat{A})_{I\times[k]} > 0$, which means:
		\begin{align}\label{y_k_lower_bound}
			\epsilon_k y_k\det\left[\boldsymbol{c}^1| \cdots| \boldsymbol{c}^{k}\right]_{I\times[k]} +\epsilon_k\displaystyle{\sum_{i=k+1}^{n}(-1)^{i+k}y_i} \det\left[\boldsymbol{c}^1| \cdots| \boldsymbol{c}^{k-1}| \boldsymbol{c}^i\right]_{I\times[k]} >0.
		\end{align}
		Since sign$\left(\det\left[\boldsymbol{c}^1| \cdots| \boldsymbol{c}^{k}\right]_{I\times[k]}\right)=\epsilon_k$, using~\eqref{y_k_lower_bound} we obtain the following lower bound for $y_k$:
		
		\begin{align}\label{y_k_lb}
			y_k>\frac{-\displaystyle{\sum_{i=k+1}^{n}(-1)^{i+k}y_i}\det\left[\boldsymbol{c}^1| \cdots| \boldsymbol{c}^{k-1}| \boldsymbol{c}^i\right]_{I\times[k]}} {\det\left[\boldsymbol{c}^1| \cdots| \boldsymbol{c}^{k}\right]_{I\times[k]}}.
		\end{align}
		
		Thus working over all such $k\times k$ contiguous minors, $\det(\widehat{A}_{I\times[k]})$ yields a finite set of lower bounds for $y_k$. Therefore, we can choose $y_k>0$ such that all $k\times k$ contiguous minors of $\widehat{A}$ containing the first column, and hence all the $k\times k$ minors, have the sign $\epsilon_k$. 
		
		Finally, to show that we can add a column on the right border of $A$, begin with the matrix $AP_n$. Using the Cauchy--Binet formula and the fact that $\det P_{k} = (-1)^{\left\lfloor\frac{k}{2}\right\rfloor}$, it follows that $AP_n$ is an SSR$(\epsilon^{\prime})$ matrix, where $\epsilon^{\prime}_i = (-1)^{\left\lfloor\frac{i}{2}\right\rfloor}\epsilon_i$ for $i\in[n]$. Next, add a column to its ultimate left to obtain an $n\times(n+1)$ SSR$(\epsilon^{\prime})$ matrix. Then multiplying the resulting matrix with $P_{n+1}$ on the right gives the desired matrix. This concludes our proof.
	\end{proof}
	
	\begin{rem}\label{Remark_line_insertion_m<=n}
		Note that Theorem~\ref{Theorem_square_line_insertion} also addresses the insertion of a column at the left (or right) border of an $m\times n$ SSR$(\epsilon)$ matrix $A$ with $m\leq n$, resulting in an SSR$(\epsilon)$ matrix. This follows from the fact that a new column can be expressed as a linear combination of the first $m$ columns of $A$. Consequently, when considering all contiguous square submatrices of the resulting matrix that include the first column, no column of $A$ beyond the $(m-1)^{th}$ column will contribute to their determinants. Similarly, we can add a column at the right border.
	\end{rem}
	The above remark raises the following question: For $m>n$, can we complete every $m\times n$ SSR$(\epsilon)$ matrix to an $m\times(n+1)$ SSR$(\epsilon^{\prime})$ matrix? The following theorem provides a positive answer.
	
	\begin{theorem}\label{Theorem_m>n_column_insertion}
		Given integers $m>n\geq1$, suppose $A$ is an $m\times n$ SSR$(\epsilon)$ matrix. Also set $\epsilon^\prime_i= \epsilon_i$ for $1\leq i \leq n$ and let $\epsilon^\prime_{n+1}$ be either $1$ or $-1$. Then we can append a new column to either border of $A$ to make it an $m\times(n+1)$ SSR$(\epsilon^{\prime})$ matrix. 
	\end{theorem}
	\begin{proof}
		Let $m>n$ and $A=[\boldsymbol{c}^1| \cdots| \boldsymbol{c}^n]$ be an $m\times n$ SSR$(\epsilon)$ matrix. Define $\widehat{A}:=[\boldsymbol{c}|A]\in\mathbb{R}^{m\times(n+1)}$, where $\boldsymbol{c} = (c_{1},\ldots,c_{m})^T =\displaystyle{\sum_{i=1}^{n}(-1)^{i-1}y_i \boldsymbol{c}^i}$ and $y_i>0$ for $i=1,\ldots,n$. Now repeating the proof of Theorem~\ref{Theorem_square_line_insertion}, we obtain $y_1,\ldots,y_n>0$ such that $\widehat{A}$ is SSR$_n(\epsilon)$. Since the column vector $\boldsymbol{c}$ is a linear combination of the remaining $n$ columns of $\widehat{A}$, all $(n+1)\times(n+1)$ minors of $\widehat{A}$ are zero. Our next goal is to perturb the first $m-n$ entries of $\boldsymbol{c}$ such that all $(n+1)\times(n+1)$ minors of $\widehat{A}$ are non-zero and have the same sign. For $1\leq s\leq m-n$, define the $m\times(n+1)$ matrix $A_s:=[\boldsymbol{e}^s| A]$, where $\boldsymbol{e}^s\in\mathbb{R}^m$ is the unit vector whose $s^{th}$ entry is 1 and the rest are zero.
		
		We now show how to perturb the entries of $\boldsymbol{c}$ to obtain an SSR matrix. The $(n+1) \times (n+1)$ minors can have sign either $\epsilon_n$ or $-\epsilon_n$. We write out the proof for $\epsilon_{n+1} = \epsilon_n$, and indicate at the very end what to do in the $\epsilon_{n+1}=-\epsilon_n$ case.
		
		If $\epsilon_{n+1} = \epsilon_n$, first perturb the entry $c_1$ of $\boldsymbol{c}$ such that $\widehat{A}$ remains SSR$_n(\epsilon)$, while its $[n+1]\times[n+1]$ minor becomes non-zero with the sign $\epsilon_n$. To proceed, suppose the moduli of all non-zero contiguous minors of $\widehat{A}$ and $A_1$ containing their first row and first column lie in the interval $[\lambda_1,\Lambda_1]$. Choose $0<\delta_1<\frac{\lambda_1}{\Lambda_1}$. Define
		\begin{align*}
			\widehat{A}_{\delta_1}:= [\boldsymbol{c}^{\delta_1}| A], ~ \text{where} ~ (\boldsymbol{c}^{\delta_1})_i:=\begin{cases}
				c_{i} + \delta_1, ~ &\text{if} ~ i=1,\\
				c_{i}, ~ &\text{otherwise}. 
			\end{cases}
		\end{align*}
		
		First, we show $\widehat{A}_{\delta_1}$ is SSR$_n(\epsilon)$. It suffices to show that all $r\times r$ leading principal minors of $\widehat{A}_{\delta_1}$ have sign $\epsilon_r$ for $r=1,\ldots,n$. Let $r\in[n]$ and consider
		\begin{align*}
			\det(\widehat{A}_{\delta_1})_{[r]\times[r]} = \det(\widehat{A})_{[r]\times[r]} + \delta_1\det(\widehat{A})_{[r]\smallsetminus\{1\}\times[r]\smallsetminus\{1\}},
		\end{align*}
where if $r=1$ then we take the determinant of the empty matrix to be $1$. Note that $\det(\widehat{A})_{[r]\smallsetminus\{1\}\times[r]\smallsetminus\{1\}} =\det(A_1)_{[r]\times[r]}$. This shows $\epsilon_r\det(\widehat{A}_{\delta_1})_{[r]\times[r]} > 0$ if $\epsilon_{r-1} = \epsilon_r$. If not, then $\epsilon_r\det(\widehat{A}_{\delta_1})_{[r]\times[r]} \geq  \lambda_1 + \delta_1(-\Lambda_1) > 0$. Thus $\widehat{A}_{\delta_1}$ is SSR$_n(\epsilon)$. Next, we show that the $[n+1]\times[n+1]$ minor of $\widehat{A}_{\delta_1}$ has the same sign as $\epsilon_n$:
		\begin{align*}
			\det(\widehat{A}_{\delta_1})_{[n+1]\times[n+1]} &= \det(\widehat{A})_{[n+1]\times[n+1]} + \delta_1\det(\widehat{A})_{[n+1]\smallsetminus\{1\}\times[n+1]\smallsetminus\{1\}} \\
			&= \delta_1\det(\widehat{A})_{[n+1]\smallsetminus\{1\}\times[n+1]\smallsetminus\{1\}}.
		\end{align*}
		
		Now, we inductively show that we can always select $\delta_k>0$ for all $k = 2,\ldots, m-n$. Assume we have already perturbed the first $k-1$ entries of the first column of $\widehat{A}$ by $\delta_1,\ldots,\delta_{k-1}$, such that it remains SSR$_n(\epsilon)$ and all its contiguous minors of size $(n+1)\times(n+1)$ containing at least one of the first $k-1$ rows have the sign $\epsilon_n$.
		
		Let the moduli of all non-zero contiguous minors of $\widehat{A}_{\delta_{k-1}}$ and $A_k$ containing their $k^{th}$ row and first column lie in the interval $[\lambda_k,\Lambda_k]$, and choose $0<\delta_k<\frac{\lambda_k}{\Lambda_k}$. Define $\widehat{A}_{\delta_k}$ as follows:
		\begin{align*}
			\widehat{A}_{\delta_k}:= [\boldsymbol{c}^{\delta_k}| A], ~ \text{where} ~ (\boldsymbol{c}^{\delta_k})_i:=\begin{cases}
				(\boldsymbol{c}^{\delta_{k-1}})_i + \delta_k, ~ &\text{for} ~~ i=k, \\
				(\boldsymbol{c}^{\delta_{k-1}})_i , ~ &\text{otherwise}.
			\end{cases}
		\end{align*}
		By the induction hypothesis, all $r\times r$ minors of $\widehat{A}_{\delta_k}$ without the $(k,1)$ entry have sign $\epsilon_r$, for each $r\in[n]$. Fix $r\in[n]$ and consider any $r\times r$ contiguous minor of $\widehat{A}_{\delta_k}$ containing its $k^{th}$ row and first column. Let the contiguous rows be indexed by $I=\{i_1,\ldots,i_r\}$, where $0<i_1<\cdots<i_r$, $i_1\in[k-r+1, k]$, and $i_l=k$ for some $l\in[r]$. Then we have the following two choices for the determinant of $(\widehat{A}_{\delta_k})_{I\times[r]}$:
		\begin{align*}
			\det(\widehat{A}_{\delta_k})_{I\times[r]} = \begin{cases}
				\det(\widehat{A}_{\delta_{k-1}})_{I\times[r]} + \delta_k\det(\widehat{A}_{\delta_{k-1}})_{I\smallsetminus\{i_l\}\times[r]\smallsetminus\{1\}}, ~ &\text{when $l$ is odd}, \\
				\det(\widehat{A}_{\delta_{k-1}})_{I\times[r]} - \delta_k\det(\widehat{A}_{\delta_{k-1}})_{I\smallsetminus\{i_l\}\times[r]\smallsetminus\{1\}}, ~ &\text{when $l$ is even}.
			\end{cases}
		\end{align*}
		
		Note that
		\begin{align*}
			\det(\widehat{A}_{\delta_{k-1}})_{I\smallsetminus\{i_l\}\times[r]\smallsetminus\{1\}} = \begin{cases}
				\det (A_k)_{I\times[r]}, ~~&\text{when $l$ is odd}, \\
				-\det (A_k)_{I\times[r]}, ~~&\text{when $l$ is even}.
			\end{cases}
		\end{align*} 
		
		If $l$ is odd, then an argument similar to the previous case gives us  $\det(\widehat{A}_{\delta_k})_{I\times[r]} = \epsilon_r$. For even $l$, we have $\epsilon_r\det(\widehat{A}_{\delta_k})_{I\times[r]} >  0$ if $\epsilon_{r-1} \neq \epsilon_r$. Otherwise $\epsilon_{r-1} = \epsilon_r$, in which case $\epsilon_r\det(\widehat{A}_{\delta_k})_{I\times[r]}\geq \lambda_k-\delta_k\Lambda_k>0$. Thus $\widehat{A}_{\delta_k}$ is SSR$_n(\epsilon)$. It only remains to show that all $(n+1)\times(n+1)$ minors of $\widehat{A}_{\delta_k}$ containing the $k^{th}$ row have the sign $\epsilon_n$. Consider the following contiguous minor:
		\begin{align*}
			\epsilon_n\det(\widehat{A}_{\delta_k})_{[k,k+n]\times[n+1]} &= \epsilon_n\det(\widehat{A}_{\delta_{k-1}})_{[k,k+n]\times[n+1]} + \epsilon_n\delta_k\det(\widehat{A}_{\delta_{k-1}})_{[k,k+n]\smallsetminus\{k\}\times[n+1]\smallsetminus\{1\}} \\
			&= \epsilon_n\delta_k\det(\widehat{A}_{\delta_{k-1}})_{[k+1,k+n]\times[n+1]\smallsetminus\{1\}} \\
			&> 0.
		\end{align*}
		Next, consider the remaining $(n+1)\times(n+1)$ contiguous minors of $\widehat{A}_{\delta_k}$ whose rows are indexed by $I = \{i_1,\ldots,i_{n+1}\}$, where $0<i_1<\cdots<i_{n+1}$, $i_1\in[k-n, k-1]$, and $i_l=k$ for some $l\in[2,n+1]$. By a similar argument as above, we have
		\begin{align*}
			\det(\widehat{A}_{\delta_k})_{I\times[n+1]} =\begin{cases}
			    \det(\widehat{A}_{\delta_{k-1}})_{I\times[n+1]} + \delta_k\det(\widehat{A}_{\delta_{k-1}})_{I\smallsetminus\{i_l\}\times[n+1]\smallsetminus\{1\}}, ~ &\text{if $l$ is odd},\\
			    \det(\widehat{A}_{\delta_{k-1}})_{I\times[n+1]} - \delta_k\det(\widehat{A}_{\delta_{k-1}})_{I\smallsetminus\{i_l\}\times[n+1]\smallsetminus\{1\}}, ~ &\text{if $l$ is even}.
			\end{cases}
		\end{align*}	
		Since sign$(\det(\widehat{A}_{\delta_{k-1}})_{I\times[n+1]}) = \epsilon_n$ by the induction hypothesis and $\delta_k>0$, the sign of $\det(\widehat{A}_{\delta_k})_{I\times[n+1]}$ is $\epsilon_n$ for $l$ odd. For even $l$, we have
		\begin{align*}
			\epsilon_n\det(\widehat{A}_{\delta_k})_{I\times[n+1]} &=	\epsilon_n\det(\widehat{A}_{\delta_{k-1}})_{I\times[n+1]} - \epsilon_n\delta_k\det(\widehat{A}_{\delta_{k-1}})_{I\smallsetminus\{i_l\}\times[n+1]\smallsetminus\{1\}} \\
			&\geq \lambda_{k} - \delta_{k}\Lambda_{k} \\
			&> 0.
		\end{align*}
		
		Proceeding inductively, we obtain an $m\times(n+1)$ completion of $A$ that is SSR$_{n+1}$. Note that $(n+1)\times(n+1)$ minors can instead be made to be of sign $-\epsilon_n$ by choosing $-\frac{\lambda_i}{\Lambda_i} < \delta_i < 0$ for $i=1,\ldots, m-n$. This concludes our proof.
	\end{proof}
	
	Using Remark~\ref{Remark_line_insertion_m<=n} and Theorem~\ref{Theorem_m>n_column_insertion}, we now prove the main theorem of this section.
	
	\begin{proof}[Proof of Theorem~\ref{A}]  \label{Proof_A}
		Assume that $A\in\mathbb{R}^{m\times n}$ is an SSR$(\epsilon)$ matrix with $m\leq n$. To extend $A$ into an SSR matrix by adding a column to its left border, use Remark~\ref{Remark_line_insertion_m<=n}. To add a row at the top border, insert a column to the left border of $A^T$ by Theorem~\ref{Theorem_m>n_column_insertion}; then transposing the resulting matrix gives the required matrix. 
		
		To include a column at the right border of $A$, first add a column to the left border of $AP_n$ using Remark~\ref{Remark_line_insertion_m<=n}. Note that $AP_n$ is an SSR$(\epsilon^{\prime})$ matrix, where $\epsilon^{\prime}_i = (-1)^{\left\lfloor\frac{i}{2}\right\rfloor}\epsilon_i$ for  $i\in[m]$. Then multiplying the resulting matrix with $P_{n+1}$ on the right yields the desired matrix.
		
		 To add a row at the bottom border of $A$, insert a column to the left border of $A^TP_n$ using Theorem~\ref{Theorem_m>n_column_insertion}. Thereafter, multiply the resulting matrix by $P_{n+1}$ and transpose it. Likewise, one can embed a row/column into $A$ at any border when $m>n$.
	\end{proof}
	
	\begin{rem}\label{Remark_SSR_p_border}
		If $p<\min\{m,n\}$, an $m\times n$ SSR$_p(\epsilon)$ matrix $A$ can be extended to another SSR$_p(\epsilon)$ matrix  by adding a line to any border. It is sufficient to show that we can add a new column to the left border of $A$ so that it remains SSR$_p(\epsilon)$. Note that the columns of $A$ after the $p^{th}$ column do not contribute to its contiguous minors that include the first column and are of size $\leq p$. Thus, we can assume without loss of generality that $A$ is an $m\times p$ SSR$(\epsilon)$ matrix. Now, repeat the proof similar to the first half of Theorem~\ref{Theorem_m>n_column_insertion} to obtain an $m\times(p+1)$ SSR$_p(\epsilon)$ matrix whose $(p+1)\times(p+1)$ minors are zero. Hence, one can add a new column/row to the border of an $m\times n$ SSR$_p(\epsilon)$ matrix such that the resulting matrix is SSR$_p(\epsilon)$.
	\end{rem}
	
	\section{Inserting lines between consecutive rows or columns}\label{Section_line_insertion_middle}
	
	This section is devoted to proving that we can place a line between any two given successive columns or rows of an SSR matrix, resulting in another SSR matrix. This result is based on the following:
	\begin{itemize}
		\item[(i)] We can always insert a line in the middle of a square SSR matrix of even order such that the resulting matrix is SSR.
		\item[(ii)] By Theorem~\ref{A}, an $m\times n$ SSR matrix $A$ can be completed to a square SSR matrix $\widehat{A}$ by adding lines to the borders of $A$ such that the desired position of line insertion is at the middle of $\widehat{A}$.
		\item[(iii)]  A submatrix of an SSR matrix is SSR.
	\end{itemize} 
Thus it suffices to show the first point above. We now do so.
	
	\begin{theorem} \label{Theorem_line_insertion_middle}
		Let $A$ be a square SSR$(\epsilon)$ matrix of even order. Then, one can always insert a line in the middle of $A$, resulting in another SSR$(\epsilon)$ matrix.
	\end{theorem}
	\begin{proof}
    By considering $A^T$ if necessary, it suffices to insert a column in $A$. Suppose $A=[\boldsymbol{c}^1| \cdots| \boldsymbol{c}^{2n}]\in\mathbb{R}^{2n\times 2n}$ is an SSR$(\epsilon)$ matrix. We begin by showing that a column can be added in the middle of $A$. Define $\widehat{A}:=[\boldsymbol{c}^1|\cdots| \boldsymbol{c}^n| \boldsymbol{c}| \boldsymbol{c}^{n+1}| \cdots| \boldsymbol{c}^{2n}]\in\mathbb{R}^{2n\times(2n+1)}$, where $\boldsymbol{c}\in\mathbb{R}^{2n}$ is to be obtained such that $\widehat{A}$ is SSR$(\epsilon)$. By Lemma~\ref{SSR_col_relation_general}, it suffices to produce $\boldsymbol{c}$ of the form:
		\begin{equation}\label{SSR_x_form}
			\boldsymbol{c} = \displaystyle{\sum_{i=1}^{n}} (-1)^{i-1}y_i \left(\boldsymbol{c}^{n-i+1} + \boldsymbol{c}^{n+i}\right), \hspace{.1cm} \text{for some choice of} \hspace{.1cm} y_1, \ldots, y_n >0.
		\end{equation}
		We aim to show that $y_1, \ldots, y_n >0$ can be chosen such that $\widehat{A}$ is an SSR$(\epsilon)$ matrix. To show this, it is enough to examine all contiguous minors of $\widehat{A}$ that include subcolumns of $\boldsymbol{c}$. Throughout this proof, we denote an arbitrary square contiguous submatrix of $\widehat{A}$ that contains a subcolumn of $\boldsymbol{c}$ as $\widetilde{A}$. Set:
		\begin{align*}
			l(\widetilde{A}) &:= \text{number of columns to the left of  a subcolumn of} ~ \boldsymbol{c} ~ \text{in $\widetilde{A}$}, \\
			r(\widetilde{A}) &:= \text{number of columns to the right of a subcolumn of} ~ \boldsymbol{c} ~ \text{in $\widetilde{A}$}, ~ \text{and} \\
			m(\widetilde{A}) &:= \min\{l(\widetilde{A}), r(\widetilde{A})\}.
		\end{align*}
		Note that the dimension of $\widetilde{A}$ is $d = l(\widetilde{A})+r(\widetilde{A}) +1\geq 2(m(\widetilde{A}))+1$, and $\widetilde{A}$ is of the form
		\begin{align}\label{A_tilde_general_form}
			\widetilde{A} = \left[\boldsymbol{c}^{n-(l(\widetilde{A})-1)}|\cdots| \boldsymbol{c}^n| \boldsymbol{c}| \boldsymbol{c}^{n+1}| \cdots| \boldsymbol{c}^{n+r(\widetilde{A})}\right]_{I\times[d]},
		\end{align}
		where $I\subseteq[2n]$ is a contiguous index set with $|I| = d$. To proceed, we draw the following conclusions:
		\begin{itemize}
			\item[(i)] $m(\widetilde{A})\in [0, n-1]$. If $m(\widetilde{A})\geq n$, then $\widetilde{A}$ must have at least $2n+1$ columns, but it does not have more than $2n$ rows. Hence, $\widetilde{A}$ is not square.
			
			\item[(ii)] For $0\leq k\leq n-1$, let $m(\widetilde{A}) = k$. In this case, the scalars $y_1,\ldots,y_{k}>0$ do not appear in the determinant of $\widetilde{A}$, while the coefficients $y_{k+1},\ldots,y_n>0$ do contribute. This follows from the multilinearity of the determinant.
		\end{itemize}
		The proof strategy is as follows: We begin by considering all submatrices $\widetilde{A}\in\mathbb{R}^{d\times d}$ satisfying $m(\widetilde{A}) = n-1$ to determine $y_n>0$ such that the sign of the determinant of $\widetilde{A}$ is $\epsilon_d$. Next, assume the values of $y_n, y_{n-1}, \ldots, y_{n-k}>0$ are known for $0\leq k\leq n-2$, ensuring the sign of the determinant of each submatrix $\widetilde{A}$ of $\widehat{A}$, with $m(\widetilde{A})= n-1, n-2, \ldots, n-k-1$, is preserved. We then find a suitable $y_{n-k-1}>0$ that does not change the sign of the determinant of any submatrix $\widetilde{A}$ of $\widehat{A}$ with $m(\widetilde{A}) = n-k-2$. 
		
		We start by letting $m(\widetilde{A})= n-1$. To obtain $y_n$, we split the proof into several cases. \smallskip
		
		\textbf{Case I.} $m(\widetilde{A}) = l(\widetilde{A}) = n-1 < r(\widetilde{A})$. \smallskip
		
		In this case, $\widetilde{A} = \left[\boldsymbol{c}^2| \cdots| \boldsymbol{c}^n| \boldsymbol{c}| \boldsymbol{c}^{n+1}| \cdots| \boldsymbol{c}^{2n}\right] \in \mathbb{R}^{2n\times2n}$. Taking its determinant, we obtain
		\begin{align*}
			\det\widetilde{A} &= \det\left[\boldsymbol{c}^2| \cdots| \boldsymbol{c}^n| \displaystyle{\sum_{i=1}^{n}} (-1)^{i-1}y_i (\boldsymbol{c}^{n-i+1} + \boldsymbol{c}^{n+i})| \boldsymbol{c}^{n+1}| \cdots| \boldsymbol{c}^{2n}\right] \\
			&= \det\left[\boldsymbol{c}^2| \cdots| \boldsymbol{c}^n| (-1)^{n-1}y_n\boldsymbol{c}^1| \boldsymbol{c}^{n+1}| \cdots| \boldsymbol{c}^{2n}\right] \\
			&= y_n\det\left[(-1)^{2(n-1)}\boldsymbol{c}^1| \boldsymbol{c}^2| \cdots| \boldsymbol{c}^n| \boldsymbol{c}^{n+1}| \cdots| \boldsymbol{c}^{2n}\right] \\
			&= y_n\det\left[\boldsymbol{c}^1| \cdots| \boldsymbol{c}^{2n}\right].
		\end{align*}
		
		\textbf{Case II.} $m(\widetilde{A}) = r(\widetilde{A}) = n-1 < l(\widetilde{A})$. \smallskip
		
		Here, $\widetilde{A} = \left[\boldsymbol{c}^1| \cdots| \boldsymbol{c}^n| \boldsymbol{c}| \boldsymbol{c}^{n+1}| \cdots| \boldsymbol{c}^{2n-1}\right] \in \mathbb{R}^{2n\times2n}$ and proceeding similar to Case I, we get
		\begin{align*}
			\det\widetilde{A} 
			= y_n\det\left[\boldsymbol{c}^1| \cdots| \boldsymbol{c}^{2n}\right].
		\end{align*}
		
		\textbf{Case III.} $m(\widetilde{A}) = l(\widetilde{A}) = r(\widetilde{A}) = n-1$. \smallskip
		
		From \eqref{A_tilde_general_form}, we have
		\begin{align*}
			\det\widetilde{A} &= \det\left[\boldsymbol{c}^2| \cdots| \boldsymbol{c}^n| \displaystyle{\sum_{i=1}^{n}} (-1)^{i-1}y_i \left(\boldsymbol{c}^{n-i+1} + \boldsymbol{c}^{n+i}\right)| \boldsymbol{c}^{n+1}| \cdots| \boldsymbol{c}^{2n-1}\right]_{I\times[d]} \\
			&= \det\left[\boldsymbol{c}^2| \cdots| \boldsymbol{c}^n| (-1)^{n-1}y_n\left(\boldsymbol{c}^1+\boldsymbol{c}^{2n}\right)| \boldsymbol{c}^{n+1}| \cdots| \boldsymbol{c}^{2n-1}\right]_{I\times[d]} \\
			&= (-1)^{n-1}y_n\det\left[\boldsymbol{c}^2| \cdots| \boldsymbol{c}^n| \boldsymbol{c}^1| \boldsymbol{c}^{n+1}| \cdots| \boldsymbol{c}^{2n-1}\right]_{I\times[d]} \\
			&+ (-1)^{n-1}y_n\det\left[\boldsymbol{c}^2| \cdots| \boldsymbol{c}^n| \boldsymbol{c}^{2n}| \boldsymbol{c}^{n+1}| \cdots| \boldsymbol{c}^{2n-1}\right]_{I\times[d]} \\
			&= y_n\left(\det\left[\boldsymbol{c}^1| \cdots| \boldsymbol{c}^{2n-1}\right]_{I\times[d]}+ \det\left[\boldsymbol{c}^2| \cdots| \boldsymbol{c}^{2n}\right]_{I\times[d]}\right).           
		\end{align*}
		Observe that in all three cases above, the sign of $\det\widetilde{A}$ remains unchanged for any $y_n>0$. Thus, assign any positive value to $y_n$. Now, assume $y_n, y_{n-1}, \ldots, y_{n-k}>0$ are chosen for some $0\leq k\leq n-2$, such that the determinants of all submatrices $\widetilde{A}\in\mathbb{R}^{d\times d}$ of $\widehat{A}$ with $m(\widetilde{A})\geq n-k-1$ have the sign $\epsilon_{d}$. We now inductively find $y_{n-k-1}>0$ such that the determinant of $\widetilde{A}\in\mathbb{R}^{d\times d}$ with $m(\widetilde{A}) = n-k-2$ has the sign $\epsilon_{d}$. Consider the following cases. \smallskip
		
		\textbf{Case I.} $m(\widetilde{A}) = l(\widetilde{A}) = n-(k+2) < r(\widetilde{A})$. \smallskip
		
		Using \eqref{A_tilde_general_form}, we obtain the following
		\begin{align} \label{caseII}
			\det\widetilde{A} &= \det\left[\boldsymbol{c}^{k+3}| \cdots| \boldsymbol{c}^n| \displaystyle{\sum_{i=1}^{n}} (-1)^{i-1}y_i (\boldsymbol{c}^{n-i+1} + \boldsymbol{c}^{n+i})| \boldsymbol{c}^{n+1}| \cdots| \boldsymbol{c}^{n+r(\widetilde{A})}\right]_{I\times[d]} \nonumber \\
			&= (-1)^{(n-k-1)-1}y_{n-k-1}\det\left[\boldsymbol{c}^{k+3}| \cdots| \boldsymbol{c}^n| \boldsymbol{c}^{k+2}| \boldsymbol{c}^{n+1}| \cdots| \boldsymbol{c}^{n+r(\widetilde{A})}\right]_{I\times[d]} \nonumber \\
			& + \displaystyle{\sum_{i=n-k}^{n}} (-1)^{i-1}y_i \det\left[\boldsymbol{c}^{k+3}| \cdots| \boldsymbol{c}^n| \boldsymbol{c}^{n-i+1}| \boldsymbol{c}^{n+1}| \cdots| \boldsymbol{c}^{n+r(\widetilde{A})}\right]_{I\times[d]} \nonumber \\
			& + \displaystyle{\sum_{i=r(\widetilde{A})+1}^{n}} (-1)^{i-1}y_i \det\left[\boldsymbol{c}^{k+3}|\cdots| \boldsymbol{c}^n| \boldsymbol{c}^{n+i}| \boldsymbol{c}^{n+1}| \cdots| \boldsymbol{c}^{n+r(\widetilde{A})}\right]_{I\times[d]}.
		\end{align}
		Note that $y_{n-k-1}$ is the unknown term in the above summation, while the multipliers $y_{i}$ for $i = n-k, \ldots, n$ in the penultimate summation in \eqref{caseII} are known from the hypothesis. Since $r(\widetilde{A}) + 1 \geq n-k$, all $y_i$ for $i=r(\widetilde{A})+1,\ldots,n$ are also known in the final summation in \eqref{caseII}. As $\widetilde{A}$ is a $d\times d$ matrix, we multiply \eqref{caseII} by $\epsilon_{d}$ and solve for $y_{n-k-1}$ using $\epsilon_{d}(\det\widetilde{A}) >0$ to obtain the following lower bound for $y_{n-k-1}$:
		\begin{align*} 
			y_{n-k-1}>\frac{-\left(\splitfrac{\displaystyle{\sum_{i=n-k}^{n}} (-1)^{i-1}y_i\det\left[\boldsymbol{c}^{k+3}| \cdots| \boldsymbol{c}^n| \boldsymbol{c}^{n-i+1}| \boldsymbol{c}^{n+1}| \cdots| \boldsymbol{c}^{n+r(\widetilde{A})}\right]_{I\times[d]}} 
			{+\displaystyle{\sum_{i=r(\widetilde{A})+1}^{n}} (-1)^{i-1}y_i\det\left[\boldsymbol{c}^{k+3}|\cdots| \boldsymbol{c}^n| \boldsymbol{c}^{n+i}| \boldsymbol{c}^{n+1}| \cdots| \boldsymbol{c}^{n+r(\widetilde{A})}\right]_{I\times[d]}}\right)}
			 {\det\left[\boldsymbol{c}^{k+2}| \boldsymbol{c}^{k+3}| \cdots| \boldsymbol{c}^n| \boldsymbol{c}^{n+1}| \cdots| \boldsymbol{c}^{n+r(\widetilde{A})}\right]_{I\times[d]}}.
		\end{align*} 
		\smallskip
		
		\textbf{Case II.} $m(\widetilde{A}) = r(\widetilde{A}) = n-(k+2) < l(\widetilde{A})$. \smallskip
		
		By a similar argument as above, one can obtain
		\begin{align*} 
			y_{n-k-1}>\frac{-\left(\splitfrac{\displaystyle{\sum_{i=l(\widetilde{A})+1}^{n}} (-1)^{i-1}y_i \det\left[\boldsymbol{c}^{n-(l(\widetilde{A})-1)}| \cdots| \boldsymbol{c}^n| \boldsymbol{c}^{n-i+1}| \boldsymbol{c}^{n+1}| \cdots| \boldsymbol{c}^{2n-(k+2)}\right]_{I\times[d]}}
			{+ \displaystyle{\sum_{i=n-k}^{n}} (-1)^{i-1}y_i\det\left[\boldsymbol{c}^{n-(l(\widetilde{A})-1)}| \cdots| \boldsymbol{c}^n| \boldsymbol{c}^{n+i}| \boldsymbol{c}^{n+1}| \cdots| \boldsymbol{c}^{2n-(k+2)}\right]_{I\times[d]}}\right)}
			{\det\left[\boldsymbol{c}^{n-(l(\widetilde{A})-1)}| \cdots| \boldsymbol{c}^n| \boldsymbol{c}^{n+1}| \cdots| \boldsymbol{c}^{2n-(k+2)}| \boldsymbol{c}^{2n-(k+1)}\right]_{I\times[d]}}.
		\end{align*}
		
		\smallskip
		
		\textbf{Case III.} $m(\widetilde{A}) = l(\widetilde{A}) = r(\widetilde{A}) = n-(k+2)$. \smallskip
		
		In this case \eqref{A_tilde_general_form} gives
		\begin{align}\label{caseIII}
			\det\widetilde{A} &= \det\left[\boldsymbol{c}^{k+3}| \cdots| \boldsymbol{c}^n| \displaystyle{\sum_{i=1}^{n}} (-1)^{i-1}y_i (\boldsymbol{c}^{n-i+1} + \boldsymbol{c}^{n+i})| \boldsymbol{c}^{n+1}| \cdots| \boldsymbol{c}^{2n-(k+2)}\right]_{I\times[d]} \nonumber \\
			&= \det\left[\boldsymbol{c}^{k+3}| \cdots| \boldsymbol{c}^n| \displaystyle{\sum_{i=n-k-1}^{n}} (-1)^{i-1}y_i(\boldsymbol{c}^{n-i+1}+\boldsymbol{c}^{n+i})| \boldsymbol{c}^{n+1}| \cdots| \boldsymbol{c}^{2n-(k+2)}\right]_{I\times[d]} \nonumber \\
			&=y_{n-k-1}\det\left[(-1)^{2(n-k-1)}\boldsymbol{c}^{k+2} | \boldsymbol{c}^{k+3}| \cdots| \boldsymbol{c}^n| \boldsymbol{c}^{n+1}| \cdots| \boldsymbol{c}^{2n-(k+2)}\right]_{I\times[d]} \nonumber \\ 
			& + y_{n-k-1}\det\left[\boldsymbol{c}^{k+3}| \cdots| \boldsymbol{c}^n| \boldsymbol{c}^{n+1}| \cdots| \boldsymbol{c}^{2n-(k+2)}| (-1)^{2(n-k-2)}\boldsymbol{c}^{2n-(k+1)}\right]_{I\times[d]} \nonumber \nonumber \\
			& + \det\left[\boldsymbol{c}^{k+3}| \cdots| \boldsymbol{c}^n| \displaystyle{\sum_{i=n-k}^{n}} (-1)^{i-1}y_i (\boldsymbol{c}^{n-i+1} + \boldsymbol{c}^{n+i})| \boldsymbol{c}^{n+1}| \cdots| \boldsymbol{c}^{2n-(k+2)}\right]_{I\times[d]}.
		\end{align}
		Using $\epsilon_d(\det\widetilde{A}) >0$ and \eqref{caseIII}, we obtain the following lower bound for $y_{n-k-1}$:
	    \begin{align*} 
			y_{n-k-1}>\frac{-\left(\splitfrac{\displaystyle{\sum_{i=n-k}^{n}} (-1)^{i-1}y_i\det\left[\boldsymbol{c}^{k+3}| \cdots| \boldsymbol{c}^n| \boldsymbol{c}^{n-i+1}| \boldsymbol{c}^{n+1}| \cdots| \boldsymbol{c}^{2n-(k+2)}\right]_{I\times[d]}}
			{+ \displaystyle{\sum_{i=n-k}^{n}} (-1)^{i-1}y_i\det\left[\boldsymbol{c}^{k+3}| \cdots| \boldsymbol{c}^n| \boldsymbol{c}^{n+i}| \boldsymbol{c}^{n+1}| \cdots| \boldsymbol{c}^{2n-(k+2)}\right]_{I\times[d]}}\right)}
			{\det\left[\boldsymbol{c}^{k+2}| \cdots| \boldsymbol{c}^n| \cdots| \boldsymbol{c}^{2n-(k+2)}\right]_{I\times[d]} + \det\left[\boldsymbol{c}^{k+3}| \cdots| \boldsymbol{c}^n| \cdots| \boldsymbol{c}^{2n-(k+1)}\right]_{I\times[d]}}.
		\end{align*}

		\smallskip
		
		In each of the above cases, one can choose $y_{n-k-1}>0$ such that the sign of $\det\widetilde{A}$ is $\epsilon_{d}$, where $\widetilde{A}\in\mathbb{R}^{d\times d}$ and satisfies $m(\widetilde{A}) = n-(k+2)$. Given that there are a finite number of contiguous square submatrices $\widetilde{A}$ of $\widehat{A}$ containing subcolumns of $\boldsymbol{c}$ with $m(\widetilde{A}) = n-k-2$, we can choose a sufficiently large $y_{n-k-1}>0$ to ensure that all such submatrices $\widetilde{A}$ preserve the sign of their minors. Hence, $\widehat{A}$ is SSR$(\epsilon)$.
		
	\end{proof}
	
	We next prove the main theorem of this section with the help of Theorems~\ref{A} and \ref{Theorem_line_insertion_middle}.
	
	\begin{proof}[Proof of Theorem~\ref{B}] 
		Given $A\in\mathbb{R}^{m\times n}$ is SSR$(\epsilon)$, here are the steps to show that a new SSR matrix can be formed by adding a column between the $k^{th}$ and $(k+1)^{th}$ columns of $A$, where $k \in [n-1]$.
		
		\textbf{Step 1.} Using Theorem~\ref{A}, add lines to the left, right, top, and bottom borders of $A$ to complete it to a square SSR matrix $B_A$, such that the $k^{th}$ and $(k+1)^{th}$ columns of $A$ lie in the center of $B_A$.
		
		\textbf{Step 2.} Now, insert a column in the middle of $B_A$ using Theorem~\ref{Theorem_line_insertion_middle} to obtain an SSR matrix $\widehat{B}_A$.
		
		\textbf{Step 3.} Remove the additional rows and columns from $\widehat{B}_A$ that were added in Step 1  to $A$.
		
		A row can be inserted between consecutive rows of $A$ by taking its transpose, following the above three steps, and then again taking the transpose of the final matrix. This completes our proof.
	\end{proof}
	
	Next, we show a similar result for an $m\times n$ SSR$_p(\epsilon)$ matrix, where $p < \min\{m,n\}$.
	
	\begin{theorem}\label{Theorem_SSR_k_middle}
		Let $m,n> p\geq1$ be integers and $\epsilon=(\epsilon_1,\ldots,\epsilon_p)$ be any sign pattern. Assume that $A$ is an $m\times n$ SSR$_p(\epsilon)$ matrix. Then, a column/row can be added between any two consecutive columns/rows of $A$ such that the resulting matrix is SSR$_p(\epsilon)$.
	\end{theorem}
	\begin{proof}
		Let $A\in\mathbb{R}^{m\times n}$ be an SSR$_p(\epsilon)$ matrix, we want to insert a column $\boldsymbol{c}$ between any two successive columns of $A$ such that it remains SSR$_p(\epsilon)$. Assume, without loss of generality, that there are $p-1$ columns on each side of $\boldsymbol{c}$. If more than $p-1$ columns are present on either side, we can ignore all columns after the $(p-1)^{th}$ column, as they do not contribute to contiguous minors of size up to $p\times p$ containing subcolumns of $\boldsymbol{c}$. Alternatively, if fewer than $p-1$ columns are present on either side, by Remark~\ref{Remark_SSR_p_border}, we can add columns at the borders of $A$ to obtain an SSR$_p(\epsilon)$ matrix with exactly $p-1$ columns to the left and to the right of $\boldsymbol{c}$. Thus, we consider
		\begin{align*}
			\widehat{A} = [\boldsymbol{c}^1| \cdots| \boldsymbol{c}^{p-1}| \boldsymbol{c}| \boldsymbol{c}^p| \cdots| \boldsymbol{c}^{2p-2}],
		\end{align*}
		and we need to make it SSR$_p(\epsilon)$. Express $\boldsymbol{c}$ as a linear combination of $2\ceil{\frac{p}{2}}$ columns of $A$ as follows:
		\begin{align*}
			\boldsymbol{c} = \displaystyle{\sum_{i=1}^{\ceil{\frac{p}{2}}}} (-1)^{i-1}y_i(\boldsymbol{c}^{p-i} + \boldsymbol{c}^{p+i-1}), ~ \text{where} ~ y_1, \ldots, y_{\ceil{\frac{p}{2}}} >0.
		\end{align*}
		We adopt the notations defined at the beginning of the proof of Theorem~\ref{Theorem_line_insertion_middle} and note the following:
		\begin{itemize}
			\item[(i)] $0\leq m(\widetilde{A})\leq \ceil{\frac{p}{2}}-1$. If $m(\widetilde{A})\geq \ceil{\frac{p}{2}}$, $\widetilde{A}$ will have at least $p+1$ columns, but we are only interested in contiguous minors of $\widehat{A}$ up to $p\times p$. \smallskip
			
			\item[(ii)] Suppose $0\leq k\leq \ceil{\frac{p}{2}}-1$ and $m(\widetilde{A}) = k$. Then the terms $y_{k+1},\ldots,y_{\ceil{\frac{p}{2}}}$ will contribute to the determinant of $\widetilde{A}$, while $y_1,\ldots,y_k$ will not. In particular, when $m(\widetilde{A}) = \ceil{\frac{p}{2}}-1$, only $y_{\ceil{\frac{p}{2}}}$ will show up in the determinant of $\widetilde{A}$.
		\end{itemize}
		
		Set $p_c := \ceil{\frac{p}{2}}$. Our goal is to find the coefficients $y_{1},\ldots,y_{p_c} > 0$ such that $\widehat{A}$ becomes SSR$_p(\epsilon)$. To show this, it suffices to consider all contiguous minors of $\widehat{A}$ that contain subcolumns of $\boldsymbol{c}$ and are up to size $p\times p$. The proof strategy is similar to that in Theorem~\ref{Theorem_line_insertion_middle}. Using point (ii), we first consider the cases where $m(\widetilde{A})= {p_c}-1$ to determine $y_{p_c}$. In this case, we have the following three possibilities:
		\[\textbf{Case I.}~~~m(\widetilde{A}) = l(\widetilde{A}) < r(\widetilde{A}), \quad
		  \textbf{Case II.}~~~m(\widetilde{A}) = r(\widetilde{A}) < l(\widetilde{A}), \quad
		  \textbf{Case III.}~~~m(\widetilde{A}) = l(\widetilde{A}) = r(\widetilde{A}).\]
		
		Simple calculations show that all three cases hold when $p$ is even, whereas only Case III is satisfied when $p$ is odd. Thus, considering all three cases for even $p$ and only Case III for odd $p$, and proceeding as in Theorem~\ref{Theorem_line_insertion_middle}, we find that any  $y_{p_c}>0$ can be chosen. 
		
		We then complete the proof by showing that $y_{p_c-k-1}>0$ can be found inductively, assuming $y_{p_c}, y_{p_c-1}, \ldots, y_{p_c - k} > 0$ are obtained for some $0\leq k\leq p_c -2$. To find $y_{p_c-k-1}>0$, we again consider the above three cases with $m(\widetilde{A}) = p_c-k-2$ to obtain an upper bound for $y_{p_c-k-1}$. Note that all three cases now hold for both even and odd $p$. This completes our proof.
	\end{proof}
	
	\section{Algorithm for constructing SSR and SSR$_p$ matrices}\label{Section_SSR_Construction}
	
	The goal of this section is to give an algorithm to construct an SSR matrix for any specified dimension and sign pattern. We mainly use Theorem~\ref{A} to carry out this construction. Here is our algorithm; in Appendix \ref{appendixA}, we will also provide a Python program to implement this algorithm. \smallskip
	
	\textbf{Algorithm 1. add$\_$col$\_$left($\mathbf{A}$)} \smallskip
	
	\textbf{Description:} This algorithm adds a new column to the left of an $m\times n$ SSR$(\epsilon)$ matrix A. The resulting $m\times(n+1)$ matrix will remain SSR$(\epsilon)$ if $m\leq n$ and become SSR$_n(\epsilon)$ with all $(n+1)\times(n+1)$ minors being zero if $m>n$.
	
	\textbf{Arguments:} $A$: An $m\times n$ SSR$(\epsilon)$ matrix.
	
	\textbf{Returns:} An $m\times(n+1)$ SSR$(\epsilon)$ matrix if $m\leq n$. Otherwise, an $m\times(n+1)$ SSR$_n(\epsilon)$ matrix whose all $(n+1)\times(n+1)$ minors are zero. \smallskip
	
	Step 1. Initialize to `row' and `col' the number of rows and columns of $A$, respectively.
	
	Step 2. Take $m = $ row. If row $<$ col, then take $n =$ row. Otherwise, $n =$ col.
	
	Step 3. Initialize $\boldsymbol{c}\in\mathbb{R}^{m}$ and $\boldsymbol{y}\in\mathbb{R}^{n}$ with zeros. Set $k=n$.
	
	Step 4. Take $(\boldsymbol{y})_n = 1$.
	
	Step 5. Update $k$ by setting $k$ to $k-1$.
	
	Step 6. For all contiguous index sets $I\subseteq[m]$ with $|I|=k$, assign the following:
	\begin{align*}
		(\boldsymbol{y})_k = \max_{I} \left\{0, \frac{-\displaystyle{\sum_{i=k+1}^{n}(-1)^{i+k}y_i}\det\left[\boldsymbol{c}^1| \cdots| \boldsymbol{c}^{k-1}| \boldsymbol{c}^i\right]_{I\times[k]}} {\det\left[\boldsymbol{c}^1| \cdots| \boldsymbol{c}^{k}\right]_{I\times[k]}}\right\} + \frac{1}{2}.
	\end{align*}
	
	Here $\boldsymbol{c}^{i}$ denotes the $i^{th}$ column of $A$.
	
	Step 7. If $k>1$, then go to Step 5. Else, go to the next step.
	
	Step 8. Assign $(\boldsymbol{y})_i = (-1)^{i+1}(\boldsymbol{y})_i$ for $i=1,\ldots,n$.
	
	Step 9. If row $\geq$ col, then set $\boldsymbol{c} = A\boldsymbol{y}$ and append $\boldsymbol{c}$ to the left of $A$, i.e., $A = [\boldsymbol{c}| A]\in\mathbb{R}^{m\times(n+1)}$.
	
	Return this $A$ and stop. Else, go to Step 10.
	
	Step 10. Assign $B = A_{[m]\times[m]}$.
	
	Step 11. Set $\boldsymbol{c} = B\boldsymbol{y}$ and add it to the left of $A$, i.e., $A = [\boldsymbol{c}| A]\in\mathbb{R}^{m\times(n+1)}$ and return $A$.
	
	\smallskip
	
	\textbf{Algorithm 2. add$\_$perturbation($\mathbf{A}$)} \smallskip
	
	\textbf{Description:} Adds a small perturbation to the $(1,1)$ entry of an $n\times n$ SSR$_{n-1}(\epsilon)$ matrix $A$, whose determinant is zero, to ensure it becomes SSR with the necessary $\epsilon_n$.
	
	\textbf{Arguments:} $A$: An $n\times n$ SSR$_{n-1}(\epsilon)$ matrix such that $\det A = 0$.
	
	\textbf{Returns:} An $n\times n$ SSR$(\epsilon)$ matrix with the required $\epsilon_n$. \smallskip
	
	Step 1. Initialize to $n$ the dimension of a square matrix $A$.
	
	Step 2. Define a matrix $B:=[\mathbf{e}^1| A_{[n]\times[n]\smallsetminus\{1\}}]\in\mathbb{R}^{n\times n}$, where $\boldsymbol{e}^1=(1,0,\ldots,0)^T\in\mathbb{R}^{n}$.
	
	Step 3. Assign the minimum and maximum absolute values of all $[r]\times [r]$ contiguous minors of matrices $A$ and $B$ to $\lambda$ and $\Lambda$, respectively, for each $r=1,\ldots,n-1$.
	
	Step 4. If sign $\epsilon_{n-1}=\epsilon_n$ then take $0 <\delta < \frac{\lambda}{\Lambda}$. Otherwise, $-\frac{\lambda}{\Lambda} <\delta <0$.
	
	Step 5. Add $\delta$ to the $(1,1)$ entry of $A$ and now return this $A$.
	
	\smallskip
	
	\textbf{Algorithm 3. SSR$\_$construction($\mathbf{m}$, $\mathbf{n}$, sign)} \smallskip
	
	\textbf{Description:} Constructs an $m\times n$ SSR matrix with the sign pattern `sign'.
	
	\textbf{Arguments:} \begin{itemize}
		\item[(i)] $m,n$: Size of the matrix.
		
		\item[(ii)] sign: A sign pattern of length $\min\{m,n\}$.
	\end{itemize}
	 
	\textbf{Returns:} An $m\times n$ SSR matrix with the sign pattern `sign'. \smallskip
	
	Step 1. Assign $i=0$, $\mathrm{min\_dim} = \min\{m,n\}$, and $\mathrm{max\_dim} = \mathrm{max}\{m,n\}$.
	
	Step 2. If the length of `sign' does not match `min$\_$dim', display the message, ``The length of the sign pattern is not correct!'' and stop. Otherwise, proceed to the next step.
	
	Step 3. If $\mathrm{min\_dim} = 1$, go to Step 4. Else, go to Step 5.
	
	Step 4. Display an $m\times n$ matrix of all ones if $\epsilon_1 = 1$ and all negative ones if $\epsilon_1=-1$ and stop.
		
	Step 5. Assign to $A$ the following $2\times2$ matrix according to the sign pattern entered by the user.
	\[\text{If}~\epsilon_1 = 1 ~\text{and}~ \epsilon_2 = 1, ~\text{then}~ A = \begin{pmatrix}
		2 & 1 \\
		1 & 1
	\end{pmatrix}, \quad \text{if} ~ \epsilon_1 = 1 ~\text{and}~ \epsilon_2 = -1, ~\text{then}~ A = \begin{pmatrix}
	1 & 1 \\
	2 & 1
	\end{pmatrix}, \]
	\[\text{if}~\epsilon_1 = -1 ~\text{and}~ \epsilon_2 = 1, ~\text{then}~ A = \begin{pmatrix}
		-2 & -1 \\
		-1 & -1
	\end{pmatrix}, \quad \text{if} ~ \epsilon_1 = -1 ~\text{and}~ \epsilon_2 = -1, ~\text{then}~ A = \begin{pmatrix}
		-1 & -1 \\
		-2 & -1
	\end{pmatrix}. \]
		
	Step 6. If $\mathrm{max\_dim} = 2$, then display $A$ and stop. Else, go to Step 7.
	
	Step 7. Set $i=i+1.$
   
    Step 8. Assign to `row' and `col' the number of rows and columns of $A$, respectively.
		
    Step 9. Call the function add$\_$col$\_$left$(A)$ described in Algorithm 1 above to add a column to $A$.
    
    Step 10. Assign $A^T$ to $A$ and update `row' and `col' with new dimensions of this $A$.
		
    Step 11. If row = col, call add$\_$perturbation$(A)$ described in Algorithm 2. Else, go to Step 12.
	
	Step 12. If $i < 2(\mathrm{min\_dim}) - 3$, then return to Step 7. Otherwise, go to Step 13.
		
	Step 13. If $m=n$, output $A$ and stop. Else, go to Step 14.
		
	Step 14. Call add$\_$col$\_$left$(A)$ $|m-n|$ times to add $|m-n|$ columns to $A$.
		
	Step 15. Return $A$ if $\mathrm{min\_dim}=m$, otherwise $A^T$. \smallskip
	
	\textbf{Algorithm 4. SSR$\_$p$\_$construction($\mathbf{m}$, $\mathbf{n}$, $\mathbf{p}$, sign)} \smallskip
	
	\textbf{Description:} Constructs an $m\times n$ SSR$_p$ matrix with the sign pattern `sign'.
	
	\textbf{Arguments:} \begin{itemize}
		\item[(i)] $m,n$: Size of the matrix.
		
		\item[(ii)] $p$: A positive integer such that $p<\min\{m,n\}$.
		
		\item[(iii)] sign: A sign pattern of length $p$.
	\end{itemize}
	
	\textbf{Returns:} An $m\times n$ SSR$_p(\epsilon)$ matrix. \smallskip
	
	Step 1. If the length of `sign' is not same as $p$, then display the message, ``The length of the sign pattern is not correct!'' and stop. Else, go to Step 2.
	
	Step 2. Call SSR$\_$construction($ p,p,$ sign) described in Algorithm 3. This will return a $p\times p$ SSR matrix $A$ with the sign pattern `sign'.
	
	Step 3. Call add$\_$col$\_$left$(A)$ $|m-p|$ times to add $|m-p|$ columns to $A$ and make it a $p\times m$ SSR matrix.
	
	Step 4. Assign $A^T$ to $A$. Now $A$ is an $m\times p$ SSR$(\epsilon)$ matrix.
	
	Step 5. Call add$\_$col$\_$left$(A)$ $|n-p|$ times to add $|n-p|$ columns to $A$ and make it an $m\times n$ SSR$_p$ matrix with the sign pattern `sign'.
	
	Step 6. Return $A$.
	
	\section*{Acknowledgments}
		We thank Apoorva Khare for a detailed reading of an earlier draft and for providing valuable feedback. We also thank the anonymous referee for providing useful comments and a reference that improved the manuscript. The first author was partially supported by INSPIRE Faculty Fellowship research grant DST/INSPIRE/04/2021/002620 (DST, Govt.~of India), and IIT Gandhinagar Internal Project: IP/IITGN/MATH/PNC/2223/25.

	\appendix
	\section{Python code}\label{appendixA}
	
	Here, we include a Python program for constructing an $m\times n$ SSR matrix of a given size and sign pattern $\epsilon\in\{\pm1\}^{\min\{m,n\}}$.
\begin{lstlisting}[caption={Importing the necessary library.}]
import numpy as np;
\end{lstlisting}
	
\begin{lstlisting}[caption={Definition of the function add$\_$col$\_$left($A$).}]
def add_col_left(A):
    row, col = A.shape; # Get the dimensions of the matrix
    m = row;
    if row < col:
       n = row;
    else:
       n = col;   
    c = np.zeros((m,1)); # Initialize the new column `c' with zeros
    y = np.zeros((n,1)); # Initialize the coefficients of `c' with zeros
    z = np.zeros((m,1)); # Initialize a temporary vector `z' with zeros
    y[n-1] = 1; # Initialize the last element of `y' to 1
    # The below for-loop finds coefficients y_{n-2}, ..., y_0
    for k in range(n-1,0,-1):
        z[k-1] = 1;
        # The below for-loop computes y_{k-1} for all contiguous minors
        for j in range(1,m-k+2):
            y[k-1] = 0;
            # This for-loop is used for summation present in the formula
            for i in range(k+1,n+1):
                C = np.zeros((k,k));
                C[:, 0:k-1] = A[j-1:j+k-1, 0:k-1];
                C[:, k-1] = A[j-1:j+k-1, i-1];
                det_A1 = np.linalg.det(C);
                det_A2 = np.linalg.det(A[j-1:j+k-1, 0:k]);
                y[k-1] = y[k-1] - (-1)**(i+k)*y[i-1]*det_A1/det_A2;
            # Evaluate the upper bound of y_{k-1} for each contiguous minor    
            if y[k-1] > 0:
               if z[k-1] > y[k-1]:
                  y[k-1] = z[k-1] + 1/2;
               else:
                  y[k-1] = y[k-1] + 1/2;
            if y[k-1] <= 0:
               y[k-1] = z[k-1] + 1/2;
            z[k-1] = y[k-1];
    # Compute the new vector `c'
    for i in range(n):
        y[i] = (-1)**i*y[i];
    if row >= col:
       c = A@y; # Matrix multiplication to compute the new column
       A = np.column_stack((c, A)); # Add the new column to the left of A
    else:
       A_copy = A.copy();
       B = A_copy[:, 0:row];
       c = B@y; # Matrix multiplication to compute the new column
       A = np.column_stack((c, A)); # Add the new column to the left of A
    return A;
\end{lstlisting}
	
\begin{lstlisting}[caption={Evaluating the maximum and minimum modulus of the matrix minors.}]
def max_min_minor(llambda, Lambda, b):
    if llambda == 0:
       llambda = b;
    if b > Lambda:
       Lambda = b;
    else:
       if b < llambda:
          llambda = b;
    return llambda, Lambda;
	\end{lstlisting}
	
\begin{lstlisting}[caption={Definition of the function  add$\_$perturbation($A$).}]
def add_perturbation(A):
    row, col = A.shape; # Get the dimensions of the matrix
    n = row; # Since the matrix is assumed to be square, row = col = n
    print('det of A is (zero): ', np.linalg.det(A))
    C = A.copy(); # Create a copy of the matrix to work with
    llambda = 0;
    Lambda = 0;
    # Iterate through all leading principal minors of size k from 1 to n-1
    for k in range(1, n):
        a = np.linalg.det(C[0:k, 0:k]);
        b = np.abs(a);
        # Update llambda and Lambda
        llambda, Lambda = max_min_minor(llambda, Lambda, b);
        if k > 1:
           B = np.zeros((k,k));
           B = C[0:k, 0:k];
           B[:,0] = 0;
           B[0,0] = 1;
           b = np.abs(np.linalg.det(B));
           llambda, Lambda = max_min_minor(llambda, Lambda, b);
    # Calculate the perturbation value (delta) based on the sign pattern
    if sign[n-1] == sign[n-2]:
       delta = -llambda/Lambda;
       i = 5;
       while delta < 0:
             delta = llambda/Lambda;
             delta = delta - (1/10)**i;
             i = i+1;
    if sign[n-1] != sign[n-2]:
       delta = llambda/Lambda;
       i = 5;
       while delta > 0:
             delta = -llambda/Lambda;
             delta = delta + (1/10)**i;
             i = i+1;
    print('delta = ', delta)
    A[0,0] = A[0,0] + delta;
    print('Det (after pert) = ', np.linalg.det(A))
    return A;
\end{lstlisting}
	
\begin{lstlisting}[caption={Definition of the function SSR$\_$construction($m$, $n$, sign).}]
def SSR_construction(m, n, sign):
    min_dim = min(m, n);
    max_dim = max(m, n);
    if len(sign) != min_dim:
       print('The length of sign pattern is not correct!')
    else:
       if min_dim == 1:
          if sign[0] == 1:
             A = np.array([[1]]);
             for i in range(max_dim-1):
                 A = np.column_stack((1, A));
             if min_dim == n:
                return A;
             else:
                return np.transpose(A);
          else:
                A = np.array([[-1]]);
                for i in range(max_dim-1):
                    A = np.column_stack((-1, A));
                if min_dim == n:
                   return A;
                else:
                   return np.transpose(A);
       else:
          if sign[0] == 1 and sign[1] == 1:
             A = np.array([[2, 1],
                           [1, 1]]);
          elif sign[0] == 1 and sign[1] == -1:
             A = np.array([[1, 1],
                           [2, 1]]);
          elif sign[0] == -1 and sign[1] == 1:
             A = np.array([[-2, -1],
                           [-1, -1]]);
          else:
             A = np.array([[-1, -1],
                           [-2, -1]]);
          if max_dim <= 2:
             return A;
          else:
             # Construct square SSR of dimension min{m, n}
             for square in range(1,2*min_dim-3):
                 row, col = A.shape;
                 A = add_col_left(A);
                 A = np.transpose(A);
                 row, col = A.shape;
                 if row == col:
                    A = add_perturbation(A);
             if m == n:
                return A;
             else:
                for s in range(np.abs(m-n)):
                    A = add_col_left(A);
                if min_dim == m:
                   return A;
                else:
                   return np.transpose(A);
\end{lstlisting}
	
\begin{lstlisting}[caption={Definition of the function SSR$\_$p$\_$construction($m$, $n$, $p$, sign).}]
def SSR_p_construction(m, n, p, sign):
    if len(sign) != p:
       print('The length of sign pattern is not correct!')
    else:
       A = SSR_construction(p, p, sign);
       for i in range(np.abs(m-p)):
           A = add_col_left(A);
       A = np.transpose(A);
       for i in range(np.abs(n-p)):
           A = add_col_left(A);
       return A;
\end{lstlisting}
	
\end{document}